\documentclass[reqno,a4paper,12pt]{amsart}
\usepackage{latexsym, amsmath, amssymb, amsthm, a4}
\usepackage{graphicx}

\newtheorem{theorem}{Theorem}[section]

\newtheorem{prop}[theorem]{Proposition}

\newtheorem{remark}[theorem]{Remark}

\renewcommand{\a}{\alpha}

\newcommand{\G}{\Gamma}

\renewcommand{\th}{\theta}

\renewcommand{\l}{\lambda}

\newcommand{\m}{\mu}

\newcommand{\x}{\xi}

\newcommand{\s}{\sigma}
\newcommand{\Si}{\Sigma}

\newcommand{\f}{\phi}

\newcommand{\h}{\chi}

\renewcommand{\o}{\omega}
\renewcommand{\O}{\Omega}

\newcommand{\R}{{\mathbb R}}

\newcommand{\kb}{{\mathbf k}}
\newcommand{\lb}{{\mathbf l}}

\newcommand{\nb}{{\mathbf n}}

\newcommand{\xb}{{\mathbf x}}
\newcommand{\yb}{{\mathbf y}}

\newcommand{\Kb}{{\mathbf K}}
\newcommand{\Lb}{{\mathbf L}}

\newcommand{\Pb}{{\mathbf P}}

\newcommand{\ssF}{\mathfrak s}
\newcommand{\SF}{\mathfrak S}

\newcommand{\Kc}{{\mathcal K}}
\newcommand{\Lc}{{\mathcal L}}
\newcommand{\Mcc}{{\mathcal M}}

\newcommand{\Rc}{{\mathcal R}}
\newcommand{\Sc}{{\mathcal S}}

\newcommand{\Bn}{\mathbf{B^\circ}}

\newcommand{\pO}{\partial\Omega}
\newcommand{\Kp}{\pmb{\Kc}}
\newtheorem{proposition}[theorem]{Proposition}
\newtheorem{corollary}[theorem]{Corollary}

\newtheorem*{theorem*}{Theorem}

\newcommand{\p}{\partial}

\newcommand{\eqnref}[1]{(\ref {#1})}

\newcommand{\Rbb}{\mathbb{R}}

\newcommand{\Kcal}{\mathcal{K}}

\newcommand{\Scal}{\mathcal{S}}

\def\Bn{{\bf n}}

\def\Bx{{\bf x}}
\def\By{{\bf y}}

\newcommand{\Ge}{\epsilon}

\newcommand{\Gl}{\lambda}

\newcommand{\GD}{\Delta}

\newcommand{\GO}{\Omega}

\newcommand{\beq}{\begin{equation}}
\newcommand{\eeq}{\end{equation}}

\numberwithin{equation}{section}
\numberwithin{figure}{section}

\begin{document}

\title[Neumann-Poincare operator]{Eigenvalues of the Neumann-Poincare operator in dimension 3: Weyl's law and geometry}
\author{Yoshihisa Miyanishi}
\address{Center for Mathematical Modeling and Data Science, Osaka University, Japan}
\email{miyanishi@sigmath.es.osaka-u.ac.jp}

\author{Grigori Rozenblum }

\address{ Chalmers University of Technology and The University of Gothenburg (Sweden); St.Petersburg State University, Dept. Math. Physics (St.Petersburg, Russia)}

\email{grigori@chalmers.se}

\dedicatory{To Volodya Maz'ya, an outstanding mathematician}
\begin{abstract}
We consider the asymptotic properties of the eigenvalues of the Neumann-Poincare (NP) operator in three dimensions. The region $\O\subset\R^3$ is bounded by a compact surface $\G=\partial \Omega$, with certain smoothness conditions imposed. The NP operator  $\Kc_{\G}$, called often `the direct value of the double layer potential', acting in $L^2(\G)$, is defined by
\begin{equation*}
    \Kc_{\G}[\psi](\xb):=\frac{1}{4\pi}\int_\G\frac{\langle \yb-\xb,\nb(\yb)\rangle}{|\xb-\yb|^3}\psi(\yb)dS_{\yb},
\end{equation*}
where $dS_{\yb}$ is the surface element and $\nb(\yb)$ is the outer unit normal vector on $\G$. The first-named author proved in \cite{Miyanishi:Weyl} that the singular numbers $s_j(\Kc_\G)$ of $\Kc_{\G}$ and the ordered moduli of its eigenvalues $\lambda_j(\Kc_\G)$ satisfy the Weyl law
\begin{equation*}
    s_j(\Kc(\G))\sim|\lambda_j(\Kc_\G)|\sim \left\{ \frac{3W(\G)-2\pi\h(\G)}{128\pi}\right\}^{\frac12}j^{-\frac12},
\end{equation*}under the condition that  $\G$ belongs to the class $C^{2, \a}$ with $\a>0,$
where $W(\G)$ and $\h(\G)$ denote, respectively, the Willmore energy and the Euler characteristic of the boundary surface $\G$. Although the NP operator is not self-adjoint (and therefore no general relations between eigenvalues and singular number exist), the ordered moduli of the eigenvalues of $\Kc_\G$ satisfy the same asymptotic relation.

Our main purpose here is to investigate the asymptotic behavior of positive and negative eigenvalues separately under the condition of infinite smoothness of the boundary $\G$. These formulas are used, in particular, to obtain certain answers to the long-standing problem of the existence or finiteness of negative eigenvalues of $\Kc_\G$. A more sophisticated estimation  allows us to give a natural  extension of the Weyl's law for the case of a smooth boundary.
\end{abstract}
\thanks{The second-named author is supported by grant RFBR  No 17-01-00668}
\maketitle

%\noindent{\footnotesize {\bf Key words}. Neumann-Poincar\'e operator, Eigenvalues, Weyl's law, Pseudo-differential operators, Willmore energy}

%\tableofcontents

%%%%%%%%%%%%%%%%%%%%%%%%%%%%
\section{Introduction and Results}
\par\hspace{5mm}

\subsection{Introduction}The Neumann--Poincar\'e (abbreviated by NP) operator is the boundary integral operator which appears naturally when solving classical boundary value problems using layer potentials. Its study (for the Laplace operator) goes back to C. Neumann \cite{Neumann-87} and H. Poincar\'{e} \cite{Poincare-AM-87} as the name of the operator suggests (this names combination was first used  by T. Carleman in his Thesis \cite{Carleman} and became conventional afterwards). If the boundary of the domain on which the NP operator is defined, is $C^{1, \alpha}$ smooth, then the NP operator
is compact. Thus the second kind Fredholm integral equation, which appears when solving Dirichlet or Neumann problems, is subject to the Fredholm index theory.

The study of   spectral properties of the NP operator was initiated by S. Zaremba, \cite{Zaremba}.
 Later on, it was proved in \cite{KPS} that the NP operator, not self-adjoint, generally, in $L^2,$ can be however realized as a self-adjoint operator in the $H^{-1/2}$- Sobolev space, provided   a new inner product  is introduced there, and therefore the NP spectrum is real,  may consist of a continuous spectrum and a discrete spectrum (and possibly  limit points of the discrete spectrum). If the domain is only Lipschitz, in particular, has corners, the corresponding NP operator does, in fact, possess a continuous spectrum (as well as eigenvalues).  If the domain has a smoother boundary, say, $C^{1, \alpha}$, then the spectrum consists of eigenvalues (converging to $0$ if there are infinitely many of them) and the point zero.

 It turns out that the properties of eigenvalues in the  two-dimensional and higher dimensional cases are quite different.  In the two-dimensional case, the rate of decay of these eigenvalues depends in a crucial way on the smoothness of the boundary, and their exact asymptotic behavior is presently known only for a few examples when these eigenvalues can be calculated explicitly. We refer to \cite{AKM2, Miyanishi:2015aa} for the progress on the convergence rate of NP eigenvalues in two dimensions.

As for the three-dimensional case, a general reasoning implies that the NP operator, being \emph{in the smooth case} a pseudodifferential operator of order $-1$, should have eigenvalues $\l_j$ having the order $j^{-\frac12}$. In fact, the first-named author  \cite{Miyanishi:2015aa} has proved the asymptotic formula for the \emph{nonincreasingly ordered } moduli of eigenvalues (see \cite{Miyanishi:Weyl} and see also Theorem \ref{main}).
However, the exact asymptotic formula for the positive and negative eigenvalues ordered separately was previously never known. One of complications stems from the circumstance that the NP operator is not self-adjoint in $L^2$.

With all this in mind, the purpose of this paper is to prove the Weyl law for the asymptotic behavior of positive and negative NP eigenvalues, separately, in three dimensions.
The reasoning is based upon the pseudodifferential representation of the NP operator and the formulas for the eigenvalue asymptotics for pseudodifferential  operators. These formulas are well known for self-adjoint operators on smooth manifolds, see \cite{BS}, and this consideration takes care of the smooth case, since it turns out that the spectral problem for the non-self-adjoint NP operator can be reduced to the one for a self-adjoint pseudodifferential operator. As it concerns the case of a boundary $\G$ of finite smoothness, exactly, of the class $C^{2,\a}$, the eigenvalue formulas of \cite{BS}, established there for smooth manifolds, cannot be applied directly and require certain additional perturbation reasoning. We present such reasoning, thus supporting the calculations in \cite{Miyanishi:Weyl}.

The coefficients in the asymptotic formulas are expressed in geometrical terms, involving the principal curvatures and the Gauss curvature of the surface. The most esthetic is the result for a convex domain, where the coefficient in the Weyl formula is expressed via the Euler characteristic and the Willmore energy of the surface.

 Especially for the case of a non-convex domain, some interesting problems arise as well. One of important by-products of our eigenvalue calculations is a certain progress in the long-standing question on the existence of negative eigenvalues of the NP operator in three dimensions. (There is a kind of inconsistency in the literature in fixing the sign in the expression for the  fundamental solution for the Laplacian and, consequently, in the integral kernel of the NP operator, therefore the question on positive and negative eigenvalues arises alternatively. The agreement we adopt follows, say, \cite{Ah2} or \cite{KPS}, so the fundamental solution is $-(4\pi |x|)^{-1}$. Besides it, we employ the second fundamental form as the inner products of the {\it outer normal derivative} and the second partial derivatives of a regular parametrization of a surface. So the principal curvatures of a sphere, for instance, are negative). The existence of at least one negative  eigenvalue was established in \cite{Ah2}, by means of explicit computations, for the case of  an oblate spheroid. Recently, the existence, again, of a negative eigenvalue was established in \cite{JiKang}  for at least one in a pair of two surfaces related by inversion, provided there exists at least one point of non-convexity. Besides it, the NP operator on  the standard torus has infinitely many negative eigenvalues as well as infinitely many positive ones, which was established, again, by means of explicit eigenvalue calculations, see \cite{AKJiKKM}.

  Our results on this latter problem are the following.  Suppose  that the surface $\G=\pO$ is infinitely smooth. First, there \emph{always} exist infinitely many positive eigenvalues. Further on, if there exists a point where at least one of principal curvatures is positive (so, the body $\O$ is not convex near this point), there exist infinitely many negative eigenvalues of the NP operator. On the other hand, if the surface is uniformly convex in the sense that the principal curvatures are everywhere negative (due to smoothness, this implies that they are separated from zero) then there may exist only finitely many negative eigenvalues. We note again that the results on the asymptotics of eigenvalues and on the negative eigenvalues are obtained under the condition of infinite smoothness of the surface while the singular numbers asymptotics is proved for surfaces of the class $C^{2,\a}$.

%%%%%%%%%%%%%%%%%%%%%%%%%%%%%%%%%%%%%%%%%%%%%%%%%%%%%%%%%%%%%%%%%%%%%%%%%%%%%%%%
\subsection{Main results}
To state the results in a more precise manner, let $\Omega$ be a $C^{1, \alpha}$ bounded region in $\Rbb^3$ (it is allowed that the boundary $\G=\pO$ consists of several connected components.) The NP operator $\Kcal_{\G} : L^2{(\G)} \rightarrow L^2{(\G)}$ is defined by
\beq\label{definition of NP operators}
\quad \Kcal_{\G}[\psi](\Bx) := \frac{1}{4\pi} \int_{\G} \frac{\langle \By-\Bx, \Bn(\By) \rangle}{|\Bx-\By|^3} \psi(\By)\; dS_{\By}
\eeq
where $dS_{\By}$ is the surface element and $\Bn(\By)$ is the outer normal unit vector to $\G$ at the boundary point $\yb$. This operator is known to be non-selfadjoint in $L^2(\G)$, unless each component of $\G$ is a sphere. However, $\Kc_{\G}$ is symmetrizable, in other words, there exists  a self-adjoint operator in the Sobolev space $H^{-\frac12}(\G)$ equipped with the norm defined as
 \begin{equation}\label{-1/2 norm}
 \|u\|_{-\frac12}=\langle -\Sc_{\G} u,u\rangle^{\frac12},
 \end{equation}
where $\Sc_{\G}$ is the single layer operator on $\G$,
\begin{equation*}
    \Sc_{\G}[\f](\xb)=-\frac{1}{4\pi}\int_{\G}|\xb-\yb|^{-1}\f(\yb)dS_{\yb}, \, \xb\in\G.
\end{equation*}

As explained above, we know already that $\Kcal_{\G}$ is a compact operator on $L^2(\G)$ and the set of its eigenvalues consists of at most countable set of real numbers, with $0$ the only possible limit point. It is also known that the eigenvalues of the NP operator lie in the interval $(-1/2, 1/2]$ and  $1/2$ is the eigenvalue corresponding to the constant eigenfunction. We denote the set of NP eigenvalues counting multiplicities by
\beq\label{spec}
\sigma_p(\Kcal_{\G})
=\{\; \pm\l_0^{\pm}(\G)\ge\pm \l_1^{\pm}(\G)\ge\dots \},
\eeq
where $\pm\l_j^{\pm}(\pO)$ are positive, resp., moduli of negative, eigenvalues of $\Kc_{\G};$ it is \emph{a priori} possible that the set of negative eigenvalues is finite or even void. Note that, by the symmetrization, there are no associated eigenfunctions and therefore the geometric multiplicity of each eigenvalue coincides with its algebraic multiplicity. The union of the sets of positive and of moduli of negative eigenvalues, again, numbered in the non-increasing order, form the sequence
\beq\label{ModeSpec}
\ssF_p(\Kcal_{\G})
=\{\;  \frac12=\l_0(\G)\ge \l_1(\G)\ge\dots \}.
\eeq
This sequence coincides with the non-increasingly ordered set of the singular numbers ($s$-numbers) of the operator $\Kc_{\G}$ \emph{considered as an operator in the Sobolev  space $H^{-1/2}(\G)$ with the above norm} (we note here that when changing the norm in a Hilbert space to an equivalent one, the $s$-numbers of an operator may change, unlike the eigenvalues). The set of $s$-numbers of the operator $\Kc_{\G}$ in $L^2(\G)$ is, generally, different from $\ssF_p(\Kcal_{\G})$ and it is denoted by
\beq \label{Snumbers}
\pmb{\m}(\Kc_{\G})=\{\;  \m_0(\G)\ge \m_1(\G)\ge\dots \}.
\eeq
As long as this does not cause confusion, we will omit the specification  $\G$ in the notations of operators, spaces, eigenvalues and singular values.

The counting functions for the sequences \eqref{spec}, \eqref{ModeSpec}, \eqref{Snumbers} are denoted  by $n_{\pm}(\l)=\#\{j: \pm \l_{j}^{\pm}>\l\}$, resp.,  $n(\l)$, $m(\l)$. Of course, $n(\l)=n_+(\l)+n_-(\l)$.

\begin{theorem}\label{mainPM} Let $\O$ be a bounded domain with $C^\infty$ boundary. Then
\begin{equation}\label{AsPM}
    \l^{\pm}_j\sim\ A_{\pm}(\G)^{\frac12}j^{-\frac12}, j\to\infty,
\end{equation}
where
\begin{equation}\label{APM}
    A_{\pm}(\G)=\frac{1}{128\pi^2}\int\limits_{\G}{dS_{\xb}}\int_0^{2\pi}[(k_1(\xb) \cos^2\th +k_2(\xb)\sin^2\th)_{\mp}]^2d\th,
\end{equation}
while
\begin{equation}\label{APMPM}
    A_+(\G)+A_-(\G)=A(\G),
\end{equation}
the latter given in \eqref{Coeff}.
Here $k_1(\xb), k_2(\xb)$ are the principal curvatures of the surface $\G$ at the point $\xb$ with the direction of the normal vector being chosen to be the exterior one; $dS_{\xb}$ is the surface element of $\G$ at $\xb$.
If the coefficient $A_+(\G)$ or $A_-(\G)$ turns out to be zero, formula \eqref{AsPM} should be understood as $\l^{\pm}_j=o(j^{-\frac12})$.
\end{theorem}
We say that the surface is \emph{almost convex} at a point  $\xb$ if $k_1(\xb), k_2(\xb)$ are non-positive (of course, if the surface is convex near a point, it is almost convex there).  By \eqref{APM}, for a surface which is almost convex at each point, $A_+(\G)=A(\G)$ and $A_-(\G)=0$. Further on, the surface is called \emph{strictly convex} at $\xb$ if, moreover, the principal curvatures of the surface at $\xb$ are negative.  For a body with smooth boundary in $\R^3$ there necessarily must exist a region in $\G$ where the surface is strictly convex. Therefore, the coefficient $A_+(\G)$ may never vanish.
\begin{corollary}\label{corconvex}
  Under the conditions of Theorem \ref{mainPM}, there exist infinitely many positive eigenvalues of the NP operator.
\end{corollary}

On the other hand,  there may exist an open subset in $\G$ where the surface is \emph{not} almost convex, in other words, where at least one of principal curvatures is positive. In this case, the coefficient $A_-(\G)$ is nonzero. By continuity, this happens even if there is just \emph{one} point on the boundary where the latter is not convex.
\begin{corollary}\label{cornonconvex} If the surface $\G$ is \emph{not} almost convex at  one point, at least, there exist infinitely many negative eigenvalues of the NP operator.
\end{corollary}

Further on, we can express the asymptotic formulas for the NP eigenvalues (and their counting functions)  using  some notions of surface geometry. To do this, we recall the definition  of  the Willmore energy  $W(\G)$:

\beq\label{definition of Willmore energy}
W(\G):=\int_{\G} H^2(\xb)\; dS_{\xb},
\eeq
where $H(\xb)$ is the mean curvature of the surface at $\xb$. If $\G$ consists of several connected components,  $W(\G)$ is the sum of the Willmore  energies of these components.

\begin{theorem}\label{main}
Let $\GO$ be a $C^{\infty}$ bounded region.
Then
\beq\label{lambdaAs}
\m_j(\Kc_{\G})\sim \lambda_j(\Kcal_{\G}) \sim A(\G)^{\frac12} j^{-1/2}\quad \text{as}\ j\rightarrow \infty,
\eeq
where
\begin{equation}\label{Coeff}
A(\G)=\frac{3W(\G) - 2\pi \chi(\G)}{128 \pi}.
\end{equation}
Here $W(\G)$ and $\chi(\G)$ denote, respectively, the Willmore energy and the Euler characteristic of the surface $\G$.
\end{theorem}
Thus, the NP operator has always the infinite rank (such question was discussed in \cite{KPS}) and the decay rate of NP eigenvalues is $j^{-1/2}$ for smooth regions. Furthermore, the integral (\ref{definition of Willmore energy}) is especially interesting because it has the remarkable property of being invariant under  M\"obius transformations of ${\Rbb}^3$, see \cite{Bl}. Thus we find that the asymptotic behavior of moduli of  NP eigenvalues and NP singular numbers is  also M\"obius invariant since the Euler characteristic is topologically invariant. We will present some further facts and applications later on (see section \ref{sec: applications}).
Theorem \ref{main} holds true even for $C^{2, \alpha}$ surfaces \cite{Miyanishi:Weyl}. This means that a kind of spectral cut-off happens: the eigenvalues of the NP operator may \emph{never} have a faster decay rate than $j^{-\frac12}$ and thus the operator may not ever  belong to a smaller Schatten class than $\Sigma_2$.
\begin{remark}\label{Rem1}Formula \eqref{lambdaAs} can be written in the equivalent form, using the counting functions:
\begin{equation}\label{LambdaCountAs}
    n(\l)\sim A(\G)\l^{-2},\, n_{\pm}(\l)\sim A_{\pm}(\G)\l^{-2},\,
     m(\l)=n_+(\l)+n_-(\l)\sim A(\G)\l^{-2}.
\end{equation}
\end{remark}
The above theorems are proved by means of finding the pseudodifferential operator representation for the NP operator or for a certain approximation of $\Kc_{\G}$ and further application of   classical results on the asymptotics of eigenvalues or $s$-numbers of negative order pseudodifferential operators.  Some further properties of these operators enable us to obtain sufficient conditions for the finiteness and for infiniteness of the set of negative eigenvalues.

\begin{theorem}\label{finitenegative} Let the body $\O$ with smooth boundary $\pO$ be strictly convex at all points. Then there may exist only finitely many negative eigenvalues of the NP operator.
\end{theorem}

To illustrate the meaning of the general results, let us consider the case $\p\GO=S^2$. It has been proved already by Poincar\'e \cite{Poincare-AM-87} that the NP eigenvalues on the two-dimensional sphere are $\frac{1}{2(2k+1)}$ for $k = 0, 1, 2\, \ldots $ and their multiplicities are $2k +1$.

It easily follows that the $j=k^2$-th eigenvalue satisfies
$$|\lambda_j(\Kcal_{S^2})|=\frac{1}{2(2k+1)} \sim \frac{1}{4} j^{-1/2}. $$
On the other hand, by Theorem \ref{main},
$|\lambda_j(\Kcal_{S^2})|\sim\frac{1}{4} j^{-1/2}$ since ${W(S^2) =4\pi}$ and $\chi(S^2)=2$.  This calculation, is, of course,  consistent with the above asymptotics of the explicit eigenvalues. Moreover, due to the convexity of the sphere, by Theorem \ref{finitenegative}, there may exist only a finite number of negative eigenvalues - in this example there are no such eigenvalues at all, so $\l^{+}_j(\Kcal_{S^2})\sim \frac{1}{4} j^{-1/2}$. In the example of an oblate spheroid, it was found in \cite{Ah2} that there exists at least one negative eigenvalue. Our result complements it by stating that there may exist only a finite set of such eigenvalues.

Another illustration concerns the surface $\pO$ being diffeomorphic to a torus  in $\R^3$. Such surface, obviously, contains points where it is not convex. Therefore, by Corollary \ref{cornonconvex}, the NP operator possesses infinitely many negative eigenvalues. For the standard metric torus, this property was recently established in \cite{AKJiKKM} by means of an explicit calculation. More generally, for any surface which is not simply connected, there are infinitely many negative eigenvalues.

Finally, suppose that the body $\O$ has some holes inside, so that the surface $\G$ consists of several connected components. In this case, the interior part of the boundary possesses necessarily a fragment where the surface is not almost convex. Therefore, for the NP operator the set of negative eigenvalues is infinite.

It is worth comparing  our decay rate results for the three dimensional NP eigenvalues, obtained here, with those for the two-dimensional NP eigenvalues. There, in the latter case, it is well known that the eigenvalues of the integral operator $\Kcal_{\G}$ are symmetric with respect to the origin. The only exception is the eigenvalue $1/2$ corresponding to the constant eigenfunction. NP eigenvalues are invariant under M\"obius transformations \cite{MR0104934}.
One of the main distinguished features here is that the decay rate  depends essentially on the smoothness of the boundary. Indeed,  it is  proved in \cite{AKM2, Miyanishi:2015aa} that if the boundary is $C^k$ smooth then
for any $\tau >-k+3/2$,
$$
|\lambda^{\pm}_j(\Kcal_{\G})| = o(j^{\tau}) \quad \text{as}\; j\rightarrow \infty.
$$
 Moreover for an \emph{analytic} boundary, the eigenvalues  have at least the exponential decay rate:
$$
|\Gl^{\pm}_{j}(\Kcal_{\G})| \le Ce^{-j\Ge} \quad \text{as}\; j\rightarrow \infty,
$$
Here $\Ge$ is the modified Grauert radius of $\p\GO$ (see \cite{AKM2} for the precise statement). For a piecewise analytic smooth $\G$ the second-named author has established recently the estimate $\l_j^{\pm}=O(e^{-cj^{\frac12}}).$ All these results contain upper estimates for eigenvalues. It is a challenge: with exception for M\"obius transformed ellipses, there exists presently not a single example of curves where the asymptotics of NP eigenvalues is  known. In particular, it is unknown whether there exist curves for which the eigenvalue decay is super-exponential, in other words, the question on the spectral cut-off property is still open here.

%%%%%%%%%%%%%%%%%%%%%%%%%%%%%%%%%%%%%%%%%%%%%%%%%%%%%%%%%%%%%%%%%%%%
The paper is organized as follows. In the next section we establish some important  relationships between  singular- and eigenvalues using Ky-Fan's theorem and the Plemelj's symmetrization principle for NP operators. In section \ref{symbol}, we introduce the approximate pseudo-differential operators for NP operators. Further on, in section 4  we show that \ref{sec: Weyl} the relations established in section 3 yield the Weyl law for NP eigenvalues.
Some applications and a discussion are provided in section \ref{sec: applications}.

V. Maz'ya made an essential contribution to the field of boundary integral equations. His paper \cite{Maz} had a great influence in the topic. We are happy to be able to contribute to the special volume dedicated to his jubilee and wish him many more years of productive activity.
%%%%%%%%%%%%%%%%%%%%%%%%%%%%%%%%%%%%%%%%%%%%%
%%%%%%%%%%%%%%%%%%%%%%%%%%%%%%%%%%%%%%%%%%%%%%%%%%%%%%%%%%%%%%%%%%%%%%%%%%
\section{Preliminaries on Schatten classes and perturbations}\label{sec: notations}

As preliminaries, we shall recall  some results on Schatten classes, used in this paper.

Let $K$ be a compact operator in a separable Hilbert space $H$. The singular values (s-numbers) $\{s_j(K) \}$ are the eigenvalues of $(K^*K)^{1/2}$ numbered in the nonincreasing order, counting multiplicities.
\beq
\sigma_{sing}(K)=\{\; s_j(K) \mid s_1(K) \geq s_2(K) \geq s_3(K) \geq\; \cdots\}.
\eeq
The algebra of operators satisfying $\Vert K\Vert_{\SF^p}^p={\rm{tr}}(K^*K)^{p/2}=\sum_{j=1}^{\infty} s_j(K)^{p}<\infty$
is called the Schatten class $\SF^p$. It is known that for $K\in \SF^p$, the singular values satisfy

\begin{equation}\label{week S}
s_j(K)=o(j^{-1/p}).
\end{equation}

The class of operators $K$ which satisfy \eqref{week S}
is called  \emph{small} weak Schatten class $\Si^0_p$, and the operators satisfying \eqref{week S} with $o(j^{-1/p})$ replaced by $O(j^{-1/p})$, form the weak Schatten class $\Si_p$ (see, e.g.,  \cite{Si}). Thus,  $ \SF^p\subset\Si_p^0\subset\Si_p$.
In particular, the Schatten class $\SF^2$ carries the individual name Hilbert-Schmidt class.
An integral operator $K$ with kernel $K(x,y)$ acting in $L^2(M)$
for some measure space $M$,  belongs to $\SF^2$ if and only if $\int\int_{M\times M}|K(x,y)|^2dxdy<\infty;$ for classes $\Si_2^0,\Si_2$ there exist no exact conditions for this kind of inclusion. An important subset in $\Si_2$ consists of operators for which the asymptotic relation
\begin{equation}\label{AbstrAs}
    s_j(K)\sim C j^{-\frac12}
\end{equation}
holds.
The following statement (by Ky Fan, \cite{KyFan}) shows that the property \eqref{AbstrAs} is stable under perturbations of $K$, belonging to the class $\Si_2^0$.

\begin{prop}\label{perturbation} If for a compact  operator $K$ the asymptotics \eqref{AbstrAs} holds and $R\in\Si_2^0$ then for the operator $K'=K+R$, the asymptotics $s_j(K')\sim C j^{-\frac12}$ is valid with the same constant $C$.
\end{prop}
%The proof of this fact can be also found, e.g., in \cite{GK}, see Theorem 2.3, Ch. 1.

As well known, starting from   H. Weyl, for \emph{self-adjoint} compact operators, a similar stability  of the  asymptotic law for separately positive and negative eigenvalues under weaker perturbations takes place as well.

\begin{prop}\label{pertSA}If for  compact  self-adjoint operators $K$, $R$ there is the  asymptotics $\l_j^{\pm}(K)\sim C_{\pm} j^{-\frac12}$ and $R\in\Si_2^0$ then for the operator $K'=K+R$, the asymptotics $\l^{\pm}_j(K')\sim C_{\pm} j^{-\frac12}$ is valid with the same constants $C_{\pm}$.
\end{prop}

 For non-selfadjoint operators, generally, \eqref{AbstrAs} does not imply that the eigenvalues or their absolute values follow a similar asymptotic law.
  %. However, in a particular case of importance for us, such implication can be established
The statement to follow enables us, however,  to pass from the asymptotics of $s$-numbers to the asymptotics of (moduli of) eigenvalues under some additional conditions.

\begin{prop}\label{almost self-adjoint1} Let $K$ be a symmetrizable compact operator, namely, there exists
an invertible operator $S$ such that $S^{-1}KS$ is self-adjoint. Assume
\begin{enumerate}
\item $S^{-1}KS-K \in \Si_2^0;$
\item $s_j(K) \sim C j^{-1/2}\quad\text{as}\ j\rightarrow \infty$.
\end{enumerate}
Then $\l_j(K) \sim s_j(K) \sim Cj^{-1/2}\quad\text{as}\ j\rightarrow \infty,$
where, recall, $\l_j(K)$ are the moduli of eigenvalues of $K$ numbered in the non-increasing order counting algebraic multiplicities.
\end{prop}
\begin{proof}
We notice that the operator $S^{-1}KS$ has the same eigenvalues of $K$.
Since the singular values of self-adjoint operator are the absolute values of eigenvalues,
the result follows from Proposition \ref{perturbation}.
\end{proof}

\section{Layer potentials as pseudo-differential operators}\label{symbol}
The single layer potential operator $-\Scal_{\G}$ is  positive, invertible, self-adjoint, and it satisfies Plemelj's symmetrization principle (also known as the Calder\'on's identity):
$$\Scal_{\G} \Kcal^{*}_{\G} =\Kcal_{\G} \Scal_{\G}.$$ Thus we can symmetrize the NP operator via single layer potentials, namely,
\begin{align}\label{symmetrized NP}
  \Kp_{\G} ={(-\Scal_{\G})}^{-1/2} \Kcal_{\G} {(-\Scal_{\G})}^{1/2}
\end{align}
is  self-adjoint in $L^2(\G)$. Our aim in this section is to find a pseudo-differential representation of the symmetrized NP operator $\Kp_\G$ \eqref{symmetrized NP}. To this purpose, we represent the single layer potential $\Scal_{\G}$ and the double layer potential $\Kcal_{\G}$ as PsDO.
Thanks to the smoothness condition, the asymptotic expansion of the integral kernel in  terms, homogeneous in $\xb-\yb$, yields a PsDO  for any local chart on the boundary manifold. The calculations in \cite{Miyanishi:Weyl} show that the principal symbol  of the NP operator, of order $-1$, equals in local co-ordinates $x',\x$
\begin{equation}\label{PrincSymbol}
p_{\Gamma}(x',\x)=-\frac{L(x') \xi_2^2 -2M(x')\xi_1 \xi_2 +N(x')\xi_1^2}{4 \det (g_{jk}) \left(\sum_{j, k}g^{jk}(x')\xi_j\xi_k\right)^{3/2}}
\end{equation}
(note especially here the minus sign in front of the fraction.) Here  $g_{jk}(x')$ denotes the metric tensor and $L(x'), M(x'), N(x')$ are the coefficients of the second fundamental form on the boundary $\G$ in the local co-ordinates  $x'$ in some domain in $\R^2$ with, recall (this is highly important), the normal vector directed to the exterior of $\O$.

The single layer potential $\Scal_{\G}$ is also PsDO of order $-1$ with principal symbol
$(-4\pi |\xi|_{x'})^{-1}.$
Here $|\xi|_{x'}=\sqrt{\sum_{j, k} g^{jk}(x')\xi_j \xi_k}$ on each local chart \cite{Agrano-book}.
It is known that the operator $-\Scal_{\G}$  is  positive and invertible, it is  elliptic, and therefore  the complex powers of $-\Sc_\G$ are pseudodifferential operators. The operator $(-\Scal_{\G})^z$ has principal symbol $(4\pi |\xi|_x)^{z}$. Therefore the principal symbol of $\Kp$ equals the product of principal symbols of $\Kc, \Sc $ and $\Sc^{-1}$, and thus,  the principal symbol of the symmetrized NP operator coincides with the one of the original NP operator. In this way, we obtained:
\begin{theorem}\label{symbol of symmtrized op}
Let $\Omega \subset{\mathbb R}^3$ be a bounded smooth region. Then the symmetrized NP operator $\Kp_\G$ is an order $-1$ pseudodifferential operator on $\G$ with principal symbol $p_{\G}(x',\x)$ given by \eqref{PrincSymbol}.
\end{theorem}
For the boundary of finite smoothness $C^{2,\a}$, the above reasoning is not valid since a convenient  symbolic calculus for operators on such manifolds is not developed. We repeat, for further reference, the local approximation result from \cite{Miyanishi:Weyl} in a convenient form:
\begin{proposition}\label{PropFinSmooth} Let $\O\subset\R^3$ be a bounded domain with $C^{2,\a}$ boundary. Let $V\subset \G$ be a coordinate patch so that  $\pi:V\to\R^2$ is a co-ordinate mapping generating the isometry $\tilde{\pi}$ of $L^2(V)$ to $L^2(U)$, $U=\pi V\subset \R^2$. If $\psi_1,\psi_2$ are bounded functions with support in $U$ then the projected operator $\tilde{\pi} \psi_1 \Kc_{\G} \psi_2\tilde{\pi}^{-1}$ is the sum of
a pseudodifferential operator in $U$ with symbol
\begin{equation}\label{symbolLocal}
 h(x',\x)=h_{\psi_1,\psi_2}(x',\x)=\tilde{\pi}\psi_1(x')\tilde{\pi}\psi_2(x')p_\G(x',\xi)
\end{equation}

  and a Hilbert-Schmidt class operator, with $p_\G$ given in \eqref{PrincSymbol}.
  This symbol is smooth in $\x$ variable and belongs to $C^\a$ in $x'$.
\end{proposition}

%%%%%%%%%%%%%%%%%%
%In \cite{Miyanishi:Weyl}, $C^{2, \alpha}$ boundaries are considered and so the PsDO representation is modulo Hilbert-Schmidt. For smooth cases the NP operator is just PsDO.
%%%%%%%%%%%%%%%%%%%%%%%%%%%%%%%%%%%%%%%%%%%%%%%%%%%%%%%%%%%%%%%%%%%%%%%%%%%%%%%%%%
\section{Local considerations and delocalization}\label{sec: Weyl}

Our aim now is to obtain the eigenvalue asymptotics for the operator ${\Kp_{\G}}$. For a smooth surface $\G=\pO$, this would follow from Theorem \ref{symbol of symmtrized op} and the basic result by Birman-Solomyak, \cite{BS}, on the asymptotics of singular numbers and eigenvalues for negative order pseudodifferential operators. For $C^{2,\a}$ surfaces, a perturbation approach was used in \cite{Miyanishi:Weyl}, with a  representation in local charts of the NP operator as a sum of a pseudodifferential operator and a Hilbert-Schmidt one, with again  further using Birman-Solomyak's result to obtain the asymptotics of moduli of eigenvalues. However, the analysis of the paper \cite{BS} shows that some more explanations are needed when applying the results of this paper to our setting.

In the original paper \cite{BS}, \emph{only} the case  of a pseudodifferential operator in a domain of the Euclidean space, with moderately regular symbol,  was considered in detail. Just a brief remark was included that the results, for a homogeneous symbol, carry over 'easily' to operators on smooth manifolds. Somehow, this statement migrated to later publications by various authors, even applied to manifolds with \emph{finite} smoothness. It turns out, however, that certain complications arise in this generalization, even for the smooth case. The authors of \cite{BS} were quite aware of these complications. In particular, when applying in \cite{BirYaf} the results of \cite{BS}  to the study of the asymptotics of scattering phases, they presented the corresponding reasoning, which turned out to be rather involved, requiring some additional technicalities. The problem consists just in the passage from  the eigenvalue asymptotics of operators in local charts to the whole manifold. We call this stage `delocalization'. In this section, we describe such delocalization, including the case of a finite smoothness, considered in \cite{Miyanishi:Weyl}. In our setting, even for a finite smoothness, the 'delocalization' is somewhat easier  than in \cite{BirYaf}, since in our  very special particular case of the NP operator on a two-dimensional manifold, the most simple Hilbert-Schmidt estimates are sufficient. In other dimensions our reasoning is also possible, however it becomes somewhat more technical, due to the absence of sharp analytical criteria for an integral operator to belong to Schatten classes other than $\SF_2$. We consider the case of a surface of the class $C^{2,\a}$; the case of an infinitely smooth surface is more simple.

We reproduce here Theorem 2 from \cite{BS}, adapted to our particular  case.
\begin{theorem}\label{BSth}Let $T$ be a pseudodifferential operator of order $-1$ with homogeneous symbol $a(x',\x)$ in a bounded domain $U\subset\R^2$, smooth in the variable $\x$ and belonging to $C^{\a}, \a>0,$ in $x'$ variable. Let also $b(x'), c(x')$ be bounded weight functions supported in $U$. Then for the operator $L=bTc$, the asymptotic formula for singular numbers  holds:
\begin{equation}\label{BS formula s}
    s_j(L)j^{\frac12}=C^{\frac12}(1+o(1)),\  C=\frac{1}{8\pi^2}\int_U\int_{S^1} (|b(x')||c(x')||a(x',\o)|)^2dx' d\o,
\end{equation}
and, for a self-adjoint operator $T$, with $b(x')=\bar{c}(x')$, the formula for eigenvalues holds:
\begin{equation}\label{BS formula l}
    \l^{\pm}_j j^{\frac12}=C_{\pm}^{\frac12}(1+o(1)), C_{\pm}=\frac{1}{8\pi^2}\int_U\int_{S^1} (|b(x')|^2a(x',\o)_{\pm})^2dx' d\o.
\end{equation}
\end{theorem}

On the  surface $\G$, we consider a finite system of disjoint open subsets $\G_m$ so that $\bigcup\overline{\G_m}=\G$. We denote by $\G_m^s$ the 'star' of $\G_m$, namely, the union of $\overline{\G_m}$ and those sets $\overline{\G_{m'}}$ for which $\overline{\G_m}\cap\overline{\G_{m'}}\neq\emptyset$. We can suppose that this decomposition is so fine that, for some choice of points $\xb_m\in\G_m$, the orthogonal projection $\pi_m$ of $\G_m$ to the tangent plane $T_m$ at  $\xb_m$ is a homeomorphism, moreover, it is a homeomorphism of $\G_m^s$. We denote by $U_m$ the range of  $\G_m^s$ under $\pi_m$, $U_m=\pi_m (\G_m^s)$. The smoothness conditions imposed on $\G$ imply that $\pi_m^{-1}$, considered as a mapping from $U_m$ to $\R^3$, is of the class $C^{2,\a}$. Further on, the mapping $\pi_m$ generates in a usual way isometries of Hilbert spaces $\widetilde{\pi_m}:L_2(U_m)\to L_2(\G_m^s)$, which are also isometries $\widetilde{\pi_m}:L_2(\pi_m(\G_m))\to L_2(\G_m)$.

Let $\h_m$ be the characteristic function of the set $\G_m$. The NP operator $\Kc$ can be represented as
\begin{equation}\label{repr}
    \Kc=\sum_{m,m'}\Kc_{m,m'},\,\, \Kc_{m,m'}=\h_m \Kc \h_{m'}.
\end{equation}

Our first statement about the singular numbers  asymptotics is the following.

\begin{proposition}\label{3A}
For any $m$, as $j\to\infty$,
\begin{equation}\label{asympLocal}
s_j(\Kc_{m,m})j^{\frac12}= C_m^{\frac12}(1+o(1)), C_m=\frac{1}{8\pi^2}\int_{\pi_m\G_m}\int_0^{2\pi}|h(x',\th)|^2dx'd\th,
\end{equation}
where $h(x',\x)$ is given by \eqref{symbolLocal} with $\psi_1=\psi_2=\tilde{\h_m}$.
\end{proposition}\begin{proof}
By means of the isometry $\widetilde{\pi_m}$ of Hilbert spaces $L^2(\G_m)$ and $L^2(U_{m})$, the operator $\Kc_{m, m}$ turns out to be unitary equivalent to the pseudodifferential operator $P_m$ with symbol determined by  \eqref{PrincSymbol} and weight functions $b=c=\tilde{\pi}_m\h_m(x')$ plus a Hilbert-Schmidt operator.  Further on, the symbol of the pseudodifferential operator $P_m$, is smooth in the $\xi$ variable and belongs to $C^{\alpha}$ in $x'$ variable. Therefore, to this operator we can apply Theorem \ref{BSth} establishing the asymptotics of singular numbers of a pseudodifferential operator in a domain in the Euclidean space,  which is given by formulas \eqref{BS formula s}.   The addition of a Hilbert-Schmidt operator, by the Weyl  inequality and Proposition \ref{almost self-adjoint1} does not change the leading term in these asymptotic formulas.
\end{proof}

The asymptotics, just found, is responsible for the diagonal terms $m=m'$ in the representation \eqref{repr}. The proposition to follow proves that the non-diagonal terms in \eqref{repr} do not affect these asymptotics.
\begin{proposition}\label{prop.as.nondia}
Let $m\ne m'$. Then for the operator $\Kc_{m,m'}=\h_m\Kc\h_{m'}$,
\begin{equation}\label{o-estimate}
   s_j(\Kc_{m,m'})=o(j^{-\frac12}).
\end{equation}
\end{proposition}
\begin{proof} We consider two cases: $\overline{\G_m}\cap\overline{\G_{m'}}\ne\emptyset$ and $\overline{\G_m}\cap\overline{\G_{m'}}=\emptyset$.

In the first case, again, the operator $\Kc_{m,m'}$ is unitary equivalent to the sum of a Hilbert-Schmidt operator and a pseudodifferential operator in $L_2(\overline{\G_m}\bigcup\overline{\G_{m'}})$ having the form
\begin{equation}\label{PsDO.Local}
    \Pb_{\G_m,\G_{m'}}u(x')=(2\pi)^{-2}\h_m(x')
    \int_{\R^2}\int\limits_{\overline{\G_m}\bigcup\overline{\G_{m'}}} P(x',\xi)e^{(x'-y')\xi}\h_{m'}(y')dy'd\xi.
\end{equation}
By Theorem \ref{BSth}, the singular numbers of the operator \eqref{PsDO.Local} have asymptotics $C_{m,m'}j^{-\frac12}$ with coefficient $C_{m,m'}$ determined by the integration of the symbol and the weight functions $\h_m,\h_m'$ over $\pi_m(\overline{\G_m}\bigcup\overline{\G_{m'}})\times S^1$. However, in this particular case, the product $\h_m(x')\h_m'(x')$ equals zero almost everywhere and the integral annules. Therefore, the coefficient in front of $j^{-\frac12}$ in the asymptotic formula for the singular numbers of the operator $\Kc_{m,m'}$ equals zero. This means that $s_j(\Kc_{m,m'})=o(j^{-\frac12})$, just what we were aiming to.

In the second case, the distance between the compact sets $\G_m$ and $\G_{m'}$ is positive, therefore the integral kernel of the NP operator is bounded on $\G_m\times\G_{m'}$. This implies, in particular, that this kernel is square integrable over $\G_m\times\G_{m'}$, and therefore the operator $\Kc_{m,m'}$ belongs to the Hilbert-Schmidt class. As we explained before, this means that $s_j(\Kc_{m,m'})=o(j^{-\frac12})$.
\end{proof}
Now we collect the local parts of $\Kc$ to obtain the spectral asymptotics.
\begin{theorem}\label{TheoremAsymptotics s} Let  $\G$ be a $C^{2,\a}$. Then the following asymptotic formula is valid:
\begin{equation}\label{AsFormGlobal s}
    \s_j(\Kc)\sim A(\G)^{\frac12}j^{-\frac12}, \, A(\G)=\frac{1}{8\pi^2}\int_{S^*\G} |p_{\G}(\xb, \omega)|^2\; dS_\xb d\omega
\end{equation}
\end{theorem}
\begin{proof}We use the block-matrix representation of the operator $\Kc$ with respect to the system of orthogonal subspaces $H_m=L^2(\G_m):$
\begin{equation}\label{split}
    \Kc=\sum_{m,m'}\h_m\Kc \h_{m'}=\sum{\Kc_{m,m'}}.
\end{equation}
The terms on the diagonal, $m=m'$ in \eqref{split} act in orthogonal subspaces, therefore  the set of s-numbers of $\sum_m \Kc_{m,m}$ is the union of such sets of $\Kc_{m,m}$, and thus the distribution function of its s-numbers  equals the sum of distribution functions for $\Kc_{m,m}$. Therefore, the asymptotic coefficients in these formulas, found in the previous lemmas, should be added up, which produces the coefficient in \eqref{AsFormGlobal s}. On the other hand, off-diagonal terms in \eqref{split} have singular numbers decaying faster than $j^{-\frac12}$ and therefore, by Propositions \ref{almost self-adjoint1}, \ref{perturbation} they do not contribute to the leading term of the asymptotics. This reasoning takes care of the asymptotics for the singular numbers.
\end{proof}
This theorem justifies the reasoning in \cite{Miyanishi:Weyl} concerning the reference to the result by Birman-Solomyak.

For the case of smooth boundary, a similar reasoning should be repeated as applied to the self-adjoint operator $\Kp$. The passage from the local asymptotic formula  \eqref{BS formula l} to the global one goes as above, just by using Proposition \ref{pertSA} and the pseudolocality property of classical pseudodifferential operators.
%%%%%%%%%%%%%%%%%%%%%%%%%%%%%%%%%%%%%%%%%%%%%%%%%%%%%%%%%%%%%%%%%%%%%%
\begin{theorem}Let $\G$ be a smooth boundary of $\O$. Then the following asymptotic formula is valid
 \begin{equation}\label{Asymp l}
 \l_j^{\pm}(\Kc)\sim A_{\pm}(\G)^{\frac12}j^{-\frac12}, \, A_{\pm}(\G)=\frac{1}{8\pi^2}\int_{S^*\G} p_{\G}(x', \omega)_{\pm}^2\; dx' d\omega.
 \end{equation}
 \end{theorem}

To finish the proof of Theorem \ref{mainPM} and Theorem \ref{main},  we calculate the positive constants $A(\G), A_\pm(\G)$. To achieve  this, we may use the isothermal charts, so that locally the metric has the form $\sum_{i,j} g_{ij}dx^i dx^j=E(x')(dx_1^2+dx_2^2)$. Then the surface element is  $dS=E(x') dx'$ and, after summing the local contributions, we obtain
\begin{align*}
A_{\pm}(\G)&=\frac{1}{8\pi^2}\int_{\p\GO} \int_{S^1} \Big[\frac{L(\xb) \xi_2^2 -2M(\xb)\xi_1 \xi_2 +N(\xb)\xi_1^2}{4 \det (g_{ij}) \{ \sqrt{\sum_{j, k} g^{jk}(\xb)\xi_j \xi_k}\}^3} \Big]_{\mp}^2\; d\xi dx' \\
&=\frac{1}{8\pi^2} \int_{\p\GO} \int_{S^1}\Big[\frac{L(\xb) \cos^2 \theta -2M(\xb)\cos \theta \sin \theta + N(\xb) \sin^2 \theta}{4 E^2(\xb) E^{-3/2}(\xb)} \Big]_{\mp}^2\; d\th dx' \\
&=\frac{1}{128\pi^2} \int_{\p\GO} \int_{S^1} \frac{ [(L(\xb) \cos^2 \theta -2M(\xb)\cos \theta \sin \theta + N(\xb) \sin^2 \theta)_{\mp}]^2}{E^2(\xb)} \; d\th dS_\xb.\end{align*}
For fixed $\xb$ one can diagonalize the last equation by an orthogonal matrix and so
\begin{align*}
A_{\pm}(\G)&=\frac{1}{128\pi^2} \int_{\p\GO} \int_{S^1} [(k_1(\xb) \cos^2 \tilde\theta + k_2(\xb) \sin^2 \tilde\theta)_{\mp}]^2 \; d\tilde\th dS_\xb.
\end{align*}
This gives us the required formula  for $A_{\pm}(\G)$ as a coordinate free representation,
For the sum $A(\G)=A_+(\G)+A_-(G)$, we can add up the above expressions for $A_{\pm}(\G)$ to obtain

\begin{align*}
A(G)&=\frac{1}{128\pi^2} \int_{\p\GO} \frac{ (\frac{3\pi}{4}L^2(\xb)+\frac{3\pi}{4}N^2(\xb)+\pi M^2(\xb) + \frac{\pi}{2}L(\xb)N(\xb)) }{E^2(\xb)} \; dS_{\xb} \\
&=\frac{1}{128\pi^2} \int_{\p\GO} \frac{ (\frac{3\pi}{4}L^2(\xb)+\frac{3\pi}{4}N^2(\xb)+\pi (L(\xb)N(\xb)-E^2(\xb)K(\xb)) + \frac{\pi}{2}L(\xb)N(\xb)) }{E^2(\xb)} \; dS_\xb \\
&=\frac{3}{512 \pi} \int_{\p\GO} \Big[\left(\frac{L(\xb)+N(\xb)}{E(\xb)} \right)^2 -\frac{4}{3} K(\xb) \Big]\; dS_{\xb} \\
&=\frac{3}{512 \pi} \int_{\p\GO} 4H^2(\xb)\; dS_{\xb} - \frac{1}{64} \chi (\p\GO) \\
&=\frac{3W(\p\GO) - 2\pi \chi(\p\GO)}{128 \pi}.
\end{align*}
Here we used the Gauss-Bonnet theorem in the integration of Gaussian curvature.

%%%%%%%%%%%%%%%%%%%%%%%%%%%%%%%%%%%%%%%%%%%%%%%%%%%%%%%%%%%%%%%%%%%%%%%%%%
\section{Applications and remarks}\label{sec: applications}

\subsection{Plasmonic eigenvalues}
The interest in the NP operator, especially in its spectral properties, has been growing rapidly recently,  due to its connection to the plasmon resonance and the anomalous localized resonance in meta materials possessing negative material characteristics, for example, dielectric constants. These resonances occur at eigenvalues and at the accumulation points of eigenvalues of the NP operator, respectively (see \cite{ACKLM, MFZ-PR-05} and references therein). The spectral nature of the NP operator is also related to stress concentration between hard inclusions \cite{BT-ARMA-13}.

As an application of our results, let us consider plasmonic eigenvalues (see e.g. \cite{Grieser} and references therein). A real number $\Ge$ is called a {\it plasmonic eigenvalue} if the following problem admits
a solution $u$ in the space $H^1(\Rbb^3)$:
\beq\label{plasmon}
\begin{cases}
\GD u =0 \quad\quad\quad\quad\, &\text{in}\  {\Rbb}^3\backslash \p\GO ,\\
u|_{-}=u|_{+} \quad\quad\quad\ &\text{on}\ \p\GO ,\\
\Ge\p_n u|_{-}= -\p_n u|_{+} \quad\ \, &\text{on}\ \p\GO.
\end{cases}
\eeq
where the subscript $\pm$ on the left-hand side respectively denotes the limit (to $\p\GO$) from the outside and inside of $\GO$.
The well-known relation \cite{AKMU} between the plasmonic eigenvalue $\Ge$ and
the NP eigenvalue $\Gl$ gives
\beq\label{plasmonic eigenvalues}
\epsilon_j - 1 =\frac{-2\lambda_j}{\lambda_j-1/2}\sim 4 A_{\pm}j^{-1/2}.
\eeq
Hence the plasmonic eigenvalues constitute  the sequence with $1$ as the limit, and the (R.H.S.) of \eqnref{plasmonic eigenvalues} gives its converging rate. Positive and negative NP eigenvalues  correspond to left limits and right limits respectively.

\subsection{(In)finiteness of the set of negative eigenvalues}
We consider now the question on negative eigenvalues of the NP operator, mentioned in the Introduction. We give here the proof of Theorem \ref{finitenegative}.

So, we suppose here that the surface $\G$ is infinitely smooth and \emph{convex}. The latter  condition means that the principal symbol $\kb(\xb,\x)$ of the NP operator $\Kc$,  considered as a pseudodifferential operator calculated in Section \ref{symbol}, is positive for all $\xb\in \pO$ and $\x: |\x|=1$, therefore, by the compactness of the cospheric bundle,
\begin{equation}\label{negative symbol}
    \kb(\xb,\x)\ge C|\x|^{-1}
\end{equation}
Therefore, the principal symbol of the \emph{self-adjoint} pseudodifferential operator ${\Kp}=\Sc^{-\frac12}\Kc\Sc^{\frac12}$ is the same, $\kb(x,\x)$. It follows that the symbol
\begin{equation}\label{SqRoot}
    \lb(x,\x)=(\kb(\xb,\x))^{-\frac12}
\end{equation}
is well defined as a smooth positive function on the cotangent bundle of  $\dot{T}^*(\pO)$ as a positive function, degree $-\frac12$ positively homogeneous.

Now we consider \emph{some} pseudodifferential operator $\Lb$ on $\pO$ with principal symbol $\lb$. This operator can be constructed in the usual way, by means of gluing together local operators with this symbol, defined by means of the Fourier transform. We denote by $\Rc$ the operator $\Lb^*\Lb$, where the adjoint operator is considered in the sense of the space $H^{-\frac12}(\G)$ with norm as in \eqref{-1/2 norm}. The operator $\Rc$, thus constructed, is a nonnegative operator in $H^{-\frac12}$, moreover, it is an elliptic operator of order $1$. Therefore, the zero subspace of $\Rc$ has finite dimension. If this subspace is nontrivial, we add to $\Rc$ the orthogonal, finite rank, projection onto this subspace, thus changing $\Rc$ by a smoothing operator. After this operation, the operator, which we still denote by $\Rc$, becomes a first order \emph{positive} elliptic operator, still, with principal symbol $\lb(x,\x)^{2}=\kb(x,\x)^{-1}$. By construction, the operator $\Rc$ is invertible and the principal symbol of the positive operator $\Kb=\Rc^{-1}$ equals, again,  $\kb(x,\x)$. So, since the operators ${\Kp}$ and $\Kb$ have the same principal symbol, their difference, $\Mcc=\tilde{\Kc}-\Kb$ is a self-adjoint pseudodifferential operator of lower order, no greater than $-2$. Therefore, $\Mcc$ has the form
\begin{equation*}
    \Mcc=\Kb Z\Kb,
\end{equation*}
with an operator $Z=\Kb^{-1}\Mcc\Kb^{-1}$, a zero order pseudodifferential operator, bounded in all Sobolev spaces $H^s(\pO)$.

Now, consider, for some $t>0$ the subspace $\Lc_t$ spanned by all eigenfunctions of $\Kb$ with eigenvalues  smaller than $t$. This subspace has finite codimension, and for $u\in\Lc_t$, we have $\|\Kb u\|^2\le t(\Kb u,u)$ (by the Spectral Theorem.) Therefore, (all scalar products and norms are in the sense of $H^{-\frac12}(\G)$) $|(\Mcc u,u)|=|(\Kb Z \Kb u,u)|=|(Z\Kb u,\Kb u)|\le \|Z\| \|\Kb u\|^2, \, u\in \Lc_t.$

So we have
\begin{equation*}
  (\Kp u,u) = (\Kb u,u)+(\Mcc u,u)\le (\Kb u,u) -t\|Z\|(\Kb u,u)=(1-t\|Z\|)(\Kb u,u).
\end{equation*}
We choose $t$ so small that $1-t\|Z\|>\frac12$. This choice gives us
\begin{equation}\label{negativity}
   (\Kp u,u)\ge \frac12   (\Kb u,u)\ge 0.
\end{equation}
In this way, we have found a subspace with finite codimension on which the quadratic form of the operator $\Kp$ is nonnegative. Again, by the Spectral Theorem, this means that the operator $\Kp,$ and together with it, the operator $\Kc$, may have only a finite number of positive eigenvalues.

On the other hand, we consider a surface which is \emph{not} almost convex. This means that somewhere at the surface, the integrand in $A_-(\G)$ in \eqref{APM} is positive, thus $A_-(\G)>0.$ This means that for the negative eigenvalues there exist the power asymptotics with a nonvanishing coefficients, which implies their infiniteness. As explained in the Introduction, the set of positive eigenvalues is always infinite.

\subsection{M\"obius (non)invariance} As described in the previous section, the asymptotic behavior  of absolute values of NP eigenvalues is related closely with the Willmore energy and the  Euler characteristics. Some applications in this direction can be found in
\cite{Miyanishi:Weyl}. The separate behavior of positive and negative eigenvalues involves more detailed structures. For instance, the asymptotic of the absolute value of NP eigenvalues is invariant under M\"obius transforms as is the case with two-dimensional NP operators. However, M\"obius transforms of ellipsoids are gourd-shaped surfaces (they may have negative curvature points). So we have infinite many negative eigenvalues on the gourd-shaped surface and its asymptotic of negative eigenvalues varies from an ellipsoid. As a result, the asymptotics of NP eigenvalues may vary under M\"obius  transforms, while Weyl's asymptotics of moduli of  eigenvalues is invariant.

%%%%%%%%%%%%%%%%%%%%%%%%%%%%%%%%%%%%%%%%%%%%%%%%%%%%%%%%%%%%%%%%%%%%%%
\subsection{Remarks}
 We should note that the results on the asymptotics of eigenvalues and on the negative eigenvalues are obtained under the condition of infinite smoothness of the surface while the singular numbers asymptotics is proved for surfaces of the class $C^{2,\a}$. We are convinced that these  results can be extended to less smooth surfaces, the ones where the asymptotic coefficients still make sense, thus to surfaces of the class $C^{1,1}$,  i.e., those which are described locally by once differentiable functions with derivatives satisfying the Lipschitz condition. The natural way to pursue this aim is to study the approximation of non-smooth surfaces by smooth ones, in the flavor of \cite{RT}.

General results on the eigenvalue asymptotics for a smooth surface can be easily extended to the multi-dimensional case, without considerable complications. What may cause certain trouble is finding an expression for the asymptotic coefficients in geometrical terms.  However, for the case of finite smoothness, the possibility of such extension is presently unclear, in particular, since the reasoning involving Hilbert-Schmidt operators should be replaced by some other Schatten classes and  presently known conditions for an integral operator to belong to such classes might be  not sufficient for our aims.

By our opinion, the results on the eigenvalue asymptotics can also be extended to the NP operator defined in a proper way with the fundamental solution for the Laplacian being  replaced by the fundamental solution for some general second order elliptic operator with smooth coefficients.

The authors mean to pursue these topics in the future.

%%%%%%%%%%%%%%%%%%%%%%%%%%%%%%%%%%%%%%%%%%%%%%%%%%%%%%%%%%%%%


\begin{thebibliography}{11}
\bibitem{Agrano-book}
\newblock M.S. Agranovich, B.Z. Katsenelenbaum, A.N. Sivov and N.N. Voitovich,
\newblock{\sl Generalized Method of Eigenoscillations in Diffraction Theory}, Wiley-VCH, Berlin, 1999.

\bibitem{Ahl}
L. V. Ahlfors,
{\it Remarks on the Neumann-Poincar\'e integral equation}, Pacific. J. Math.,
{\bf 2} (1952), 271--280.

%\bibitem{Ah1}
%J. F. Ahner,
%{\it Some spectral properties of an integral operator in potential theroy},
%Proc. Edinburgh Math. Soc., {\bf 29} (1986), 405--411.

\bibitem{Ah2}
J. F. Ahner,
{\it On the eigenvalues of the electrostatic integral operator II}, J. Math. Anal. Appl.,
{\bf 181} (1994), 328--334.

%\bibitem{AA}
%J. F. Ahner and R. F. Arenstorf,
%{\it On the eigenvalues of the electrostatic integral operator}, J. Math. Anal. Appl.,
%{\bf 117} (1986), 187--197.

\bibitem{ACKLM}
H.~{Ammari}, G.~{Ciaolo}, H.~{Kang}, H.~{Lee} and G.~W.~{Milton}.
\newblock {Spectral theory of a Neumann-Poincar\'{e}-type operator and analysis of cloaking due to anomalous localized resonance},
\newblock {\em Arch. Ration. Mech. An.} {\bf 208} (2013), 667--692.

\bibitem{AKL}{H. Ammari, H. Kang, and H. Lee},
\newblock
{\sl Layer potential techniques in spectral analysis},
\newblock Mathematical Surveys and Monographs, {\bf 153} American Math. Soc., Providence RI, (2009).

\bibitem{AK-JMAA-16}
K.~{Ando} and H.~{Kang},
\newblock Analysis of plasmon resonance on smooth domains using spectral properties of the {N}eumann--{P}oincar{\'e} operator,
\newblock J. Math. Anal. Appl. {\bf 435}(1) (2016), 162--178.

%\bibitem{AKM}
%K. Ando, H. Kang and Y. Miyanishi,
%\newblock {Elastic Neumann-Poincar\'{e} operators on three dimensional smooth domains: Polynomial compactness and spectral structure},
%\newblock Int. Math. Res. Notices, rnx258, (2017).

\bibitem{AKM2}
K. Ando, H. Kang, Y. Miyanishi,
\newblock Exponential decay estimates of the eigenvalues for the Neumann--Poincar\'e operator on analytic boundaries in two dimensions,
\newblock To appear in J. Integral Equations, arXiv:1606.01483.

\bibitem{AKMU}
K. Ando, H. Kang, Y. Miyanishi and E. Ushikoshi,
\newblock {The first Hadamard variation of Neumann--Poincar\'e eigenvalues},
\newblock To appear in Proc. Amer. Math. Soc., arXiv:1805.02414.

%\bibitem{Andreev}
%A. S. Andreev,
%\newblock {Asymptotics of the spectrum of compact pseudodifferential operators in a
%Euclidean domain},
%\newblock Mat. Sbornik {\bf 137} (1988), 203--223; Math. USSR Sbornik {\bf 65} (1990).

%\bibitem{Bhatia} R. Bhatia,
%\newblock Some inequalities for norm ideals,
%\newblock Comm. Math. Phys., {\bf 111} (1987), 33--39.

\bibitem{BS}
M. Birman and M. Solomyak,
\newblock Asymptotic behavior of the spectrum of pseudodifferential operators with
anisotropically homogeneous symbols,
\newblock Vestnik Leningrad Univ. {\bf 13} (1977), 13--21; English translation in Vestin. Leningr. Univ. Math. {\bf 10}, 237--247.

\bibitem{BirYaf} M. Birman, D. Yafaev,  Asymptotic behavior of the spectrum of the scattering matrix. (Russian) Boundary value problems of mathematical physics and related questions in the theory of functions, 13. Zap. Nauchn. Sem. Leningrad. Otdel. Mat. Inst. Steklov. (LOMI) \textbf{110} (1981), 3-29. English translation in: J. Sov. Math. \textbf{25} (1984), 793-814.

\bibitem{Bl}
W. Blaschke, {\it Voresungen \"{U}ber Differentialgeometrie III}, Berlin: Springer (1929)

%\bibitem{BM}
%J. Blumenfeld and W. Mayer,
%\newblock \"Uber poincar\'e fundamental funktionen,
%\newblock Sitz. Wien. Akad. Wiss., Math.-Nat. Klasse {\bf 122}, Abt. IIa (1914), 2011--2047.

\bibitem{BT-ARMA-13} E. Bonnetier and F. Triki,
\newblock On the spectrum of Poincar\'e variational problem for two close-to-touching inclusions in 2D,
\newblock Arch. Ration. Mech. An. {\bf 209} (2013), 541--567.

\bibitem{Carleman} T. Carleman, Uber das Neumann-Poincar\'{e}sche problem f\"{u}r ein gebiet mit ecken,
Almquist and Wiksells, Uppsala, 1916.

%\bibitem{CH}
%\newblock R. Courant and D. Hilbert,
%\newblock Methods of Mathematical Physics I, Wiley-Interscience, (1953).

%\bibitem{CM}
%{R. R. Coifman} and {Y. Meyer},
%\newblock{Au dul\`{a} des op\'{e}rateurs pseudodiff\'{e}rentiels,}
%\newblock{Asterisque} {\bf 57} (1978), 1--185.

%\bibitem{DK}
%D. Deturck and L. Kazdan,
%\newblock Some regularity theorems in Riemannian geometry,
%\newblock Ann. Sci. \'{E}c. Norm. Sup\'{e}r. (4), {\bf 14} (3) (1981), 249--260.

%\bibitem{Dostanic}
%M. Dostani\'c,
%\newblock A theorem of Ky-Fan type,
%\newblock Matemati\u{c}ki Vesnik, {\bf 47} (1995), 7--10.

%\bibitem{Fredholm-03} E. I. Fredholm, Sur une classe d'equations fonctionnelles, Acta Mathematica, {\bf 27} (1903), 365--390.

%\bibitem{Gil1}
%M. I. Gil',
%\newblock Lower bounds for eigenvalues of Schatten-Von Neumann operators,
%\newblock J. Inequal. Pure Appl. Math., {\bf 8}(3) (2007), Art. 66.

%\bibitem{Gil2}
%M. I. Gil',
%\newblock Operator functions and localization of spectra,
%\newblock Lectures Notes in Mathematics, vol.1830, Springer-Verlag, Berlin, 2003.

%\bibitem{GK}
%I. C. Gohberg and M. G. Krein,
%\newblock Introduction to the Theory of Linear Nonselfadjoint Operators,
%\newblock Trans. Mathem. Monographs, {\bf 18}: (1969) Amer. Math. Soc., Providence, R.I.

\bibitem{Grieser}
\newblock D. Grieser, The plasmonic eigenvalue problem,
\newblock Rev. Math. Phys. {\bf 26} (2014), 1450005.

%\bibitem{Grubb}
%G. Grubb,
%\newblock Spectral asymptotics for nonsmooth singular green operators,
%\newblock Comm. P. D. E., {\bf 39}: (2014), 530--573.

%\bibitem{HKL-AIHP-17} J. Helsing, H. Kang and M. Lim,
%\newblock Classification of spectra of the Neumann--Poincar\'{e} operator on planar domains with corners by resonance,
%\newblock Ann. I. H. Poincare-AN, to appear., arXiv:1603.03522, 2016.

\bibitem{JiKang} Y. Ji, H.Kang, A concavity condition for existence of a negative Neumann-Poincar\'{e} eigenvalue in three dimensions. 	 \texttt{arXiv:1808.10621}

\bibitem{AKJiKKM}
K. Ando, Y. Ji, H. Kang, D. Kawagoe, and Y. Miyanishi
\newblock Spectral structure of the Neumann--Poincar\'e operator on tori,
\newblock arXiv:1810.09693

%\bibitem{KKLSY-JLMS-16} H. Kang, K. Kim, H. Lee, J. Shin and S. Yu,
%\newblock Spectral properties of the Neumann--Poincar\'e operator and uniformity of estimates for the conductivity equation with complex coefficients,
%\newblock J. London Math. Soc. (2) {\bf 93} (2016), 519--546.

\bibitem{KLY}
H.~{Kang}, M.~{Lim} and S.~{Yu},
\newblock {Spectral resolution of the Neumann-Poincar\'{e} operator on intersecting disks and analysis of plasmon resonance},
\newblock arXiv:1501.02952, 2015.

%\bibitem{Kato} T. Kato,
%\newblock Variation of discrete spectra,
%\newblock Comm. Math. Phys., {\bf 111}: (1987), 501--504.

%\bibitem{MR0222317}
%O. D. Kellogg,
%\newblock Foundations of Potential Theory,
%\newblock Dover, New York, 1953.

\bibitem{KPS}
D.~{Khavinson}, M.~{Putinar}, and H.S.~{Shapiro},
\newblock Poincar\'e's variational problem in potential theory.
\newblock Arch. Ration. Mech. Anal. {\bf 185}(1) (2007), 143--184.

%\bibitem{Ko}{\sc H. K\"onig},
%\newblock Eigenvalue distribution of compact operators.
%\newblock Operator theory: Advances and Applications, {\bf 16}. Birkh\"auser Verlag, Basel, (1986).

\bibitem{KyFan}Ky Fan,
Maximum properties and inequalities for the eigenvalues of completely continuous operators.
\emph{Proc. Nat. Acad. Sci.}, U. S. A. \textbf{37}, (1951), 760--766.
%\bibitem{LR}
%R. Langevin and H. Rosenberg,
%\newblock On curvature integrals and knots.
%\newblock Topology, 15:405416, (1976).

\bibitem{MFZ-PR-05} I. D. Mayergoyz, D. R. Fredkin and Z. Zhang,
\newblock{Electrostatic (plasmon) resonances in nanoparticles},
\newblock{Phys. Rev. B}, 72 (2005), 155412.

\bibitem{MN}
F. C. Marques and A. Neves,
\newblock Min-Max theory and the Willmore conjecture,
\newblock Anal. Math. {\bf 179} (2014), 683--782.

\bibitem{Ma}
E. Martensen,
\newblock A spectral property of the electrostatic integral operator,
\newblock J. Math. Anal. Appl., {\bf 238} (1999), 551--557.

\bibitem{Maz}V. Maz'ya. Boundary Integral Equations, Itogi Nauki i Techniki, Fundamental Directions, vol. 27,  Analysis-4, Viniti, 1988, pp. 131--228.

%\bibitem{Mc}
%C. A. McCarthy,
%\newblock $C_p$,
%\newblock Israel J. Math., {\bf 5} (1967), 249--271.



\bibitem{MR2263683}
G.W. Milton and N.-A.P. Nicorovici,
\newblock On the cloaking effects associated with anomalous localized resonance,
\newblock Proc. R. Soc. A {\bf 462} (2006), 3027--3059.

\bibitem{Miyanishi:2015aa}
Y.~{Miyanishi} and T.~{Suzuki},
\newblock Eigenvalues and eigenfunctions of double layer potentials,
\newblock Trans. Amer. Math. {\bf 369} (2017), 8037--8059

\bibitem{Miyanishi:Weyl}Y.~{Miyanishi},
\newblock Weyl's law for the eigenvalues of the Neumann--Poincar\'e operators in three dimensions: Willmore energy and surface geometry,
\newblock arXiv:1806.03657

\bibitem{Neumann-87} C. Neumann,
\newblock \"Uber die Methode des arithmetischen Mittels, Erste and zweite Abhandlung,
\newblock Leipzig 1887/88, in Abh. d. Kgl. S\"achs Ges. d. Wiss., IX and XIII.

%\bibitem{Pl}
%J. Plemeij,
%\newblock Potentialtheoretische Untersuchungen, Preisschriften der Fiirstlich Jablonowskischen
%\newblock Gesellschaft zu Leipzig, Teubner-Verlag, Leipzig, (1911).

%\bibitem{PP-JAM-14}
%K. Perfekt and M. Putinar,
%\newblock Spectral bounds for the Neumann--Poincar\'e operator on planar domains with corners,
%\newblock J. Anal. Math. {\bf 124} (2014), 39--57.

\bibitem{PP-arXiv} K. Perfekt and M. Putinar,
\newblock The essential spectrum of the Neumann--Poincar\'e operator on a domain with corners,
\newblock  Arch. Ration. Mech. Anal. \textbf{223} (2017), no. 2, 1019--1033.

\bibitem{Poincare-AM-87}
H. Poincar\'{e},
\newblock La m\'{e}thode de Neumann et le probl\`{e}me de Dirichlet,
\newblock Acta Math. {\bf 20} (1897), 59--152.

%\bibitem{RS}
%M. Reed and B. Simon,
%\newblock Methods of Modern Mathematical Physics, I. Functional Analysis,
%\newblock Academic Press, New York, (1972).

%\bibitem{Res}{Y. G. Reshetnyak},
%\newblock \emph{Two-Dimensional Manifolds of Bounded Curvature, in Geometry IV. Nonregular Riemannian Geometry, } English Translation by E. Primrose,
%\newblock Springer Encyclopedia of Mathematical Sciences, {\bf 70} (1993), 3--163.

\bibitem{Ri2}
S. Ritter,
\newblock The spectrum of the electrostatic integral operator for an ellipsoid,
\newblock in ``Inverse Scattering and Potential Problems in Mathematical Physics,"
(R.F.Kleinman, R.Kress, and E.Marstensen, Eds.),
Lang, Frankfurt/Bern, (1995), 157--167.

\bibitem{RT}
G. Rozenblum and G. Tashchiyan,
\newblock{Eigenvalue asymptotics for potential type operators on Lipschitz surfaces},
\newblock Russ. J. Math. Phys. {\bf 13} (3), (2006), 326--339.

\bibitem{MR0104934}
M. Schiffer,
\newblock The Fredholm eigenvalues of plane domains,
\newblock Pacific J. Math. {\bf 7} (1957), 1187--1225.

%\bibitem{Sh}
%M. Schiffer,
%\newblock Fredholm eigenvalues and conformal mapping,
%\newblock Autovalori e autosoluzioni, C.I.M.E. Summer Schools {\bf 27}, Springer (2011), 203--234.

%\bibitem{S}
%H. S. Shapiro,
%\newblock The Schwarz function and its generalization to higher dimensions,
%\newblock University of Arkansas Lecture Notes in the Mathematical Sciences, {\bf 9}.
%A Wiley Interscience Publication. John Wiley and Sons, Inc., New York, (1992).

%\bibitem{Shubin} M. A. Shubin,
%\newblock Pseudodifferential Operators and Spectral Theory (Second Edition),
%\newblock Springer, (2001).

\bibitem{Si}
B. Simon,
\newblock Trace ideals and their applications, 2nd ed.,
\newblock Amer. Math. Soc. (2005).

\bibitem{SW}
O. Steinbach and W. L. Wendland,
\newblock On C.Neumann's method for second-order elliptic systems in domains with non-smooth boundaries,
\newblock J. Math. Anal. Appl.{\bf 262} (2001), 733--748.

\bibitem{Ta}
M. E. Taylor,
\newblock Tools for PDE: Pseudodifferential Operators, Paradifferential Operators, and Layer Potentials, \newblock Mathematical Surveys and Monographs,  {\bf 81} American Math. Soc., Providence RI, (2000).

%\bibitem{Tr}
%F. G. Tricomi,
%\newblock Integral Equations,
%\newblock Wiley, New York (1957).

\bibitem{Verch-JFA-84} {G.C. Verchota},
\newblock Layer potentials and boundary value problems for Laplace's equation in Lipschitz domains,
\newblock J. Funct. Anal. {\bf 59} (1984), 572--611.

\bibitem{Wh}
J. H. White,
\newblock A global invariant of conformal mappings in space,
\newblock Proc. Amer. Math. Soc. {\bf 38} (1973), 162--164.

%\bibitem{Yosida}
%K. Yosida,
%\newblock {\sl Functional analysis}, Sixth edition,
%\newblock Springer-Verlag, Berlin-New York, 1980.

\bibitem{Zaremba}S. Zaremba, Les fonctions fondamentales de M. Poincar\'{e} et la m\'{e}thode de Neumann
pour une fronti\'{e}re compos\'{e} de polygones curvilignes, Journal de Math\'{e}matiques Pures
et Appliqu\'{e}es 10 (1904), 395--444.
%%%%%%%%%%%%%%%%%%%%%%%%%%%%%%%%%%%%%%%%%%%%%%%%%%%%%%%%%%%%%%%%%%
\end{thebibliography}
\end{document}